\documentclass[11pt,a4paper]{article}
\usepackage[T1]{fontenc}
\usepackage[utf8]{inputenc}
\usepackage{lmodern,microtype}
\usepackage[margin=27mm]{geometry}
\usepackage{amsmath,amssymb,amsthm}
\usepackage{booktabs,array}
\usepackage{tikz}
\usepackage[hidelinks]{hyperref}
\hypersetup{pdftitle={Counting with Hiccup Sequences},pdfauthor={Benoit Cloitre},
 pdfsubject={Hiccup sequences, nested recurrences, and combinatorial interpretations}}
\numberwithin{equation}{section}
\newtheorem{theorem}{Theorem}[section]
\newtheorem{proposition}[theorem]{Proposition}
\newtheorem{corollary}[theorem]{Corollary}

\theoremstyle{definition}
\newtheorem{definition}[theorem]{Definition}
\theoremstyle{remark}

\newcommand{\N}{\mathbb N}
\newcommand{\Nzero}{\mathbb N_0}
\newcommand{\im}{\operatorname{im}}
\newcommand{\rc}{\operatorname{rc}}
\newcommand{\seqnum}[1]{\href{https://oeis.org/#1}{\texttt{#1}}}
\newcommand{\floor}[1]{\left\lfloor #1\right\rfloor}
\newcommand{\ceil}[1]{\left\lceil #1\right\rceil}
\title{Counting with Hiccup Sequences}
\author{Beno\^{i}t Cloitre\\[1mm]
 \small ORCID \href{https://orcid.org/0009-0001-6778-153X}{0009-0001-6778-153X}}
\date{October 7, 2026}

\begin{document}
\maketitle

\begin{abstract}
The $(0,x,1,2)$-hiccup sequence starts at $x$ and increases by $1$
when the current index has already appeared as a value, and by $2$
otherwise. We prove directly that for $x\geq3$ the gaps between its
missing values are $1,2,3,\ldots$, each repeated $x-2$ times, and we
derive closed formulas, an explicit error bound, and a
generating function in terms of Ramanujan's theta functions. The
counting functions of the complements satisfy nested recurrences.
The sequences of missing values are the partial complements of
polygonal numbers. In particular, the
$(x-2)$-fold composition of the enumeration of the positive integers
that are not $x$-gonal is the hiccup sequence minus $1$. We extend the
nested recurrence and a tree interpretation to every miss increment
$z\geq2$ and every initial value $x\geq2$. Further interpretations
involve polyomino vertices, monochromatic pairs, and, by a theorem of
Bickle, Ramsey core numbers.
\end{abstract}

\section{Introduction}

A hiccup sequence chooses its next increment by testing whether an
integer has already appeared among its values. We introduced these
sequences in 2003 and stated the rule with $j=0$ in the introduction to
our joint paper with Sloane and Vandermast~\cite[Section 1]{CSV}, where
we deferred its study to a separate work that was never completed.
The present paper carries out part of that study. Several of these
sequences entered the OEIS in 2003, some submitted by Sloane. Fokkink
and Joshi~\cite{FJ} give a general treatment of hiccup sequences by
morphisms. We use their notation.

Throughout the paper, $\N=\{1,2,\ldots\}$ and
$\Nzero=\{0,1,2,\ldots\}$. Sequences are indexed from $1$ unless
stated otherwise.

\begin{definition}
Let $j,x\in\Nzero$ and let $y,z$ be distinct positive integers.
The $(j,x,y,z)$-hiccup sequence has $a(1)=x$ and, for $n\geq2$,
\begin{equation}
 a(n)=a(n-1)+
 \begin{cases}
 y,&n-j\in\{a(1),\ldots,a(n-1)\},\\
 z,&n-j\notin\{a(1),\ldots,a(n-1)\}.
 \end{cases}
\end{equation}
The first case is a \emph{hit} and the second a \emph{miss}.
\end{definition}

We take $j=0$ and $y=1$ throughout. We also take $z=2$, except in
Propositions~\ref{prop:generalnested} and~\ref{prop:tree}. The sequence
starting at $4$ begins
\[
4,6,8,9,11,12,14,15,16,18,\ldots,
\]
where the second and third steps are misses and the fourth is a hit,
because $4=a(1)$ is already a value. Its missing values
\[
1,2,3,5,7,10,13,17,21,\ldots
\]
have gaps $1,1,2,2,3,3,4,4,\ldots$, and we show that the same pattern,
with each gap repeated $x-2$ times, holds for every initial value $x\geq3$.

For $x=1$ and $x=2$, we have $a(n)=n+1$ for $n\geq2$. The only
missing integer is $2$ when $x=1$ and $1$ when $x=2$. For $x\geq3$,
put $q=x-2$ and write $a_q$ for the $(0,q+2,1,2)$-hiccup sequence.
Let $c_q(1)<c_q(2)<\cdots$ enumerate its missing positive integers,
and set
\begin{equation}
 C_q(n)=\#\{j\geq1:c_q(j)\leq n\},\qquad C_q(0)=0.
\end{equation}
The first terms of $a_q$ for $1\leq q\leq4$ are
\[
\begin{array}{c|rrrrrrrr}
x\backslash n&1&2&3&4&5&6&7&8\\ \hline
3&3&5&6&8&9&10&12&13\\
4&4&6&8&9&11&12&14&15\\
5&5&7&9&11&12&14&15&17\\
6&6&8&10&12&14&15&17&18
\end{array}
\]

Theorem~\ref{thm:complement} shows that
\[
a_q(n)=n+q+C_q(n),\qquad c_q(j+1)-c_q(j)=\ceil{j/q},
\]
which leads to a closed formula for every initial value. For $1\leq q\leq7$ it
simplifies to
\[
a_q(n)=\ceil{\bigl(\sqrt n+\sqrt{q/2}\bigr)^2},
\]
and for $q=8$ the same expression must be increased by $1$ at the odd
squares. The counting function $C_q$ satisfies a nested recurrence.
The identity $C_q(a_q(n))=C_q(n)+q$ also gives closed formulas for
the iterates of $a_q$.

The orbit of $1$ under $a_q$ is the sequence of $(q+2)$-gonal numbers.
Let $b_q$ enumerate the positive integers that are not $(q+2)$-gonal,
that is, the set
\[
\N\setminus\left\{q\binom m2+m:m\geq1\right\},
\]
and write $u^{\circ t}$ for the $t$-fold composition of a map $u$, with
$u^{\circ0}$ the identity. In Section~\ref{sec:cloitre-nested} we prove
that
\[
b_q^{\circ q}(n)=a_q(n)-1.
\]
For example, $b_2(b_2(n))+1=a_2(n)$,
where $b_2$ enumerates the nonsquares.

We treated the cases $x=3,4,5$ and their counting interpretations in
an earlier paper~\cite[Theorems 5.1, 5.2, and 5.3]{Cloitre}.
The OEIS records the closed formulas for $x=3$ and $x=4$. Kimberling and
Brown~\cite{KB} study the sequences $c_q$ as partial complements of
polygonal numbers. The repeated gaps also follow from the morphism of
Fokkink and Joshi~\cite[Lemma 5, Equation (6)]{FJ}. We recalled the connection with
the meta-Fibonacci sequences of Ruskey and Deugau~\cite[Section 7]{RD}
in our earlier paper~\cite[Theorem 6.1]{Cloitre}.

We give a direct counting proof for all initial values and use it to
relate these descriptions. We derive the radical formulas and their
exact range of validity, the iteration identity above, and a tree
model for miss increments $z\geq2$ and initial values $x\geq2$.
Section~\ref{sec:series} gives the generating function of the
complement in terms of Ramanujan's theta functions.
Sections~\ref{sec:counting} and~\ref{sec:ramsey} give the
combinatorial interpretations.

\section{Complements}

\begin{theorem}\label{thm:complement}
For every integer $q\geq1$, the $(0,q+2,1,2)$-hiccup sequence $a_q$
has an infinite complement. For $n,j\geq1$,
\begin{align}
 a_q(n)&=n+q+C_q(n),\label{eq:aC}\\
 c_q(j+q)&=c_q(j)+j+q-1,\label{eq:comprec}\\
 c_q(j+1)-c_q(j)&=\ceil{j/q}.\label{eq:gaps}
\end{align}
\end{theorem}

\begin{proof}
Since every increment is positive and $a_q(n)\geq n+q+1$, an integer
$n$ that occurs as a value occurs before step $n$. For $n\geq2$, step
$n$ is therefore a miss exactly when $n$ lies in the complement, which
begins with $1,2,\ldots,q+1$. Summing the increments gives
\[
 a_q(n)=q+2+(n-1)+(C_q(n)-1),
\]
and this is~\eqref{eq:aC}, also for $n=1$.

At a missing index $c_q(j)\geq2$ the sequence jumps by $2$ and skips
the single integer
\[
 a_q(c_q(j))-1=c_q(j)+q+j-1.
\]
Every missing integer larger than $q+1$ is skipped in this way, and the
skipped integers increase with $j$. The complement is infinite, since
if it were finite, its largest element would produce a larger one. The
integer skipped at $c_q(j)$ is preceded in the complement by the $q+1$
initial missing integers and by the $j-2$ integers skipped at
$c_q(2),\ldots,c_q(j-1)$, so its rank is $j+q$. This
is~\eqref{eq:comprec} for $j\geq2$, and for $j=1$ the identity reduces
to $c_q(q+1)=q+1$.

The differences $d_j=c_q(j+1)-c_q(j)$ satisfy $d_{j+q}=d_j+1$
by~\eqref{eq:comprec}, and $d_1=\cdots=d_q=1$,
so~\eqref{eq:gaps} follows.
\end{proof}

Summing the gaps, we find that if $j-1=qm+r$ with $m\geq0$ and
$0\leq r<q$, then
\begin{equation}\label{eq:cclosed}
 c_q(j)=1+\frac{qm(m+1)}2+r(m+1).
\end{equation}

\subsection{Closed formulas}

Put
\begin{equation}
 B_q(m)=1+\frac{qm(m-1)}2\qquad(m\geq1).
\end{equation}
By~\eqref{eq:cclosed}, the missing integers in $[B_q(m),B_q(m+1))$ are
the $q$ numbers
\[
 B_q(m),\ B_q(m)+m,\ \ldots,\ B_q(m)+(q-1)m,
\]
which gives the following count.

\begin{corollary}\label{cor:allclosed}
For $q,n\geq1$, let
\[
 M=\floor{\frac{q+\sqrt{q^2+8q(n-1)}}{2q}}.
\]
Then
\begin{align}
 C_q(n)&=q(M-1)+1+\floor{\frac{n-B_q(M)}M},\\
 a_q(n)&=n+qM+1+\floor{\frac{n-B_q(M)}M}.
\end{align}
\end{corollary}

\begin{proof}
The choice of $M$ means that $B_q(M)\leq n<B_q(M+1)$. There are
$q(M-1)$ missing integers below $B_q(M)$ and
$1+\lfloor(n-B_q(M))/M\rfloor$ in $[B_q(M),n]$. This gives the
formula for $C_q(n)$. The formula for $a_q(n)$ follows
from~\eqref{eq:aC}.
\end{proof}

\subsection{Radical formulas}

Inverting the quadratic~\eqref{eq:cclosed} generally leaves a
dependence on the residue of $j-1$ modulo $q$. For $q\leq7$ this
dependence disappears after rounding.

\begin{proposition}\label{prop:radical}
For $1\leq q\leq7$ and $n\geq1$,
\begin{align}
 C_q(n)
 &=\ceil{\sqrt{2qn}-\frac q2}
 =\floor{\frac{2-q+\sqrt{q^2+8q(n-1)}}2},\label{eq:radical}\\
 a_q(n)&=\ceil{\bigl(\sqrt n+\sqrt{q/2}\bigr)^2}.
 \label{eq:hiccupradical}
\end{align}
For every $q\geq8$, each expression in~\eqref{eq:radical} fails for
some $n$.
\end{proposition}

\begin{proof}
Put
\[
 \tau_q(j)=1+\frac{(j-1)(j+q-1)}{2q},\qquad
 y_q(j)=\frac{(2j-2+q)^2}{8q}.
\]
Writing $j-1=qm+r$, with $0\leq r<q$, gives
\begin{equation}\label{eq:roundingerror}
 c_q(j)-\tau_q(j)=\frac{r(q-r)}{2q},\qquad
 y_q(j)+1-c_q(j)=\frac{(q-2r)^2}{8q}.
\end{equation}
The floor expression in~\eqref{eq:radical} first reaches the value $j$
at $n=\lceil\tau_q(j)\rceil$, and this equals $c_q(j)$ for all $j$ if
and only if $r(q-r)<2q$ for every residue $r$. As
$\max_r r(q-r)=\lfloor q^2/4\rfloor$, this means $q\leq7$.

For these $q$ the number $\sqrt{2qn}-q/2$ is positive, and its ceiling
is at least $j$ if and only if $n>y_q(j)$. The second identity
in~\eqref{eq:roundingerror} shows that
\[
0\leq y_q(j)+1-c_q(j)\leq q/8<1,
\]
so $\lfloor y_q(j)\rfloor+1=c_q(j)$. Thus the ceiling expression and
$C_q(n)$ are both at least $j$ exactly when $n\geq c_q(j)$. Then~\eqref{eq:aC}
yields~\eqref{eq:hiccupradical}. For $q\geq8$ the ceiling expression
is at most $0$ at $n=1$, whereas $C_q(1)=1$.
\end{proof}

The floor expression first fails at $n=4$ for $q=8$ and at $n=3$ for
$q\geq9$. For $q=8$ the ceiling expression needs a correction at the odd squares.

\begin{corollary}\label{cor:q8}
For every $n\geq1$,
\begin{equation}\label{eq:q8}
 a_8(n)=\ceil{(\sqrt n+2)^2}
       +\mathbf1_{\{n\text{ is an odd square}\}}.
\end{equation}
\end{corollary}

\begin{proof}
For $q=8$ the second quantity in~\eqref{eq:roundingerror} is at most
$1$, with equality if and only if $r=0$. Hence the ceiling expression
reaches the value $j$ at $c_8(j)$ when $r\neq0$, and at $c_8(j)+1$
when $r=0$. Since $c_8(8m+1)=1+4m(m+1)=(2m+1)^2$, the ceiling
expression is smaller than $C_8$ by $1$ exactly at the odd squares.
Adding $n+8$ gives~\eqref{eq:q8}.
\end{proof}

The sequences $a_1$, $a_2$, and $a_3$ are \seqnum{A080036},
\seqnum{A080037}, and \seqnum{A080353} in the OEIS~\cite{OEIS}. Sloane
entered them in March 2003, from our work for~\cite{CSV}. The first two entries have an extra term at
index $0$, and their titles give the closed formulas
$a_1(n)=n+1+\operatorname{round}(\sqrt{2n})$ and
$a_2(n)=n+2+\floor{\sqrt{4n-3}}$, where $\operatorname{round}$ denotes
the nearest integer. Tiberio records the empirical formula
$a_2(n)=\ceil{(1+\sqrt n)^2}$ in \seqnum{A080037}.
Proposition~\ref{prop:radical} proves it, since for $q=2$ it gives
\[
a_2(n)=n+1+\ceil{2\sqrt n}
      =n+2+\floor{\sqrt{4n-3}}.
\]
Krizek notes in \seqnum{A080036} that the complement of $a_1$ is
\seqnum{A000124}, which is the case $q=1$ of~\eqref{eq:cclosed}.
For $q=3$, Corollary~\ref{cor:allclosed} proves the formula
stated in our earlier paper~\cite[Theorem 5.3]{Cloitre}.

\subsection{Error bounds}

\begin{proposition}
For $q,n\geq1$, define $\theta_q(n)$ by
\[
a_q(n)=n+\sqrt{2qn}+\frac q2+\theta_q(n).
\]
Then
\begin{equation}
0\leq\theta_q(n)
\leq1+\sqrt{2qn+\frac{q(q-8)}4}-\sqrt{2qn}
\leq1+\frac{\max\{0,q(q-8)\}}{8\sqrt{2qn}}.
\end{equation}
For $q\leq7$ this gives $0\leq\theta_q(n)<1$. For $q=8$, both endpoints
of $0\leq\theta_8(n)\leq1$ are attained infinitely often.
\end{proposition}

\begin{proof}
Let $j=C_q(n)$, so that $\theta_q(n)=j+q/2-\sqrt{2qn}$
by~\eqref{eq:aC}. By~\eqref{eq:roundingerror},
$\tau_q(j)\leq c_q(j)\leq\tau_q(j)+q/8$. Combined with
$n\leq c_q(j+1)-1$, this gives
\[
n\leq\tau_q(j+1)+\frac q8-1
 =\frac{(j+q/2)^2}{2q},
\]
that is, $j+q/2\geq\sqrt{2qn}$, which is the lower bound. In the other
direction, $\tau_q(j)\leq c_q(j)\leq n$ implies
\[
j+\frac q2\leq1+\sqrt{2qn+\frac{q(q-8)}4},
\]
which is the middle bound. For $q<8$ the difference of square roots is
negative, so $\theta_q(n)<1$, and for $q\geq8$ rationalizing it yields
the last bound. Finally, Corollary~\ref{cor:q8} shows that $\theta_8$
equals $0$ at the even squares and $1$ at the odd squares.
\end{proof}

In particular, $a_q(n)=n+\sqrt{2qn}+O_q(1)$. Thus the initial value
determines the leading term of $a_q(n)-n$.

\section{Generating functions}\label{sec:series}

Let $w_q(n)$ be $0$ if $n$ is missing from $a_q$ and $1$ otherwise. By
Theorem~\ref{thm:complement},
\begin{equation}
w_q=(0)^q(01)^q(011)^q(0111)^q\cdots
   =\prod_{m\geq0}(01^m)^q,
\end{equation}
where the product denotes concatenation. Fokkink and
Joshi~\cite[Lemma 5, Equation (6)]{FJ} obtain the same word by
iterating the morphism
\[
\beta\longmapsto\beta0^q1,\qquad
0\longmapsto01,\qquad1\longmapsto1
\]
from $\beta$ and then replacing $\beta$ by $0$. This gives a second
proof of the gap formula.

\subsection{Missing values}

Since $\sum_{j\geq1}\lceil j/q\rceil z^j=z/((1-z)(1-z^q))$, summing
the gaps gives
\begin{equation}
G_q(z)=\sum_{j\geq1}c_q(j)z^{j-1}
 =\frac1{1-z}+\frac{z}{(1-z)^2(1-z^q)}.
\end{equation}
For $q=3$ this can be written
\begin{equation}\label{eq:partitionGF}
1+\sum_{j\geq1}c_3(j)z^j
=\frac{1+z^5}{(1-z)(1-z^2)(1-z^3)},
\end{equation}
which is the generating function of \seqnum{A008748}, whose term of
index $0$ is the constant term $1$. The formula recorded there reads
\[
c_3(j)=1+\floor{j(j+1)/6}\qquad(j\geq1),
\]
and the complement of $a_3$ begins $1,2,3,4,6,8,10,13,16,19,\ldots$.
Expanding~\eqref{eq:partitionGF} shows that $c_3(j)$ is the number of
partitions of $j$ into parts $1$, $2$, and $3$, with at most one
additional part $5$. Thus $a_3$ enumerates the positive integers that
are not values of \seqnum{A008748}. This proves the complement relation
suggested by De Vlieger in \seqnum{A080353}, after correcting the
sequence number A008478 to A008748.

\subsection{Theta functions}

By~\eqref{eq:cclosed}, the indicator series of the complement is
\begin{equation}\label{eq:missingGF}
H_q(z)=\sum_{j\geq1}z^{c_q(j)}
 =\sum_{m\geq0}\sum_{r=0}^{q-1}
 z^{1+qm(m+1)/2+r(m+1)}.
\end{equation}
Dividing by $1-z$ counts the missing integers up to $n$, so
by~\eqref{eq:aC}
\begin{equation}\label{eq:aGF}
\sum_{n\geq1}C_q(n)z^n=\frac{H_q(z)}{1-z},\qquad
\sum_{n\geq1}a_q(n)z^n
=\frac{z}{(1-z)^2}+\frac{qz+H_q(z)}{1-z}.
\end{equation}
Because the exponents in~\eqref{eq:missingGF} are quadratic in $m$,
$H_q$ can be written with Ramanujan's general theta function
\[
f(a,b)=\sum_{\nu\in\mathbb Z}
a^{\nu(\nu+1)/2}b^{\nu(\nu-1)/2},\qquad |ab|<1,
\]
and with $\psi(z)=f(z,z^3)=\sum_{m\geq0}z^{m(m+1)/2}$, following
Berndt~\cite[Chapter 16, p.~34 and Entry 22]{Berndt}. The terms with
indices $\nu$ and $1-\nu$ in $f(1,z)$ coincide, so
$\psi(z)=\tfrac12f(1,z)$.

\begin{proposition}
For every integer $q\geq1$,
\begin{equation}\label{eq:theta}
H_q(z)=\frac z2\sum_{r=0}^{q-1}f(z^r,z^{q-r})
       -\frac{(q-1)z}{2},
\end{equation}
as formal power series in $z$, each theta function being expanded in
powers of $z$, and hence for $|z|<1$.
\end{proposition}

\begin{proof}
For $0\leq r<q$ put
\[
S_r(z)=\sum_{\nu\geq1}z^{q\nu(\nu-1)/2+r\nu}.
\]
The substitution $\nu=m+1$ in~\eqref{eq:missingGF} gives
$H_q(z)=z\sum_{r=0}^{q-1}S_r(z)$. On the other hand
$S_0(z)=\psi(z^q)=\tfrac12f(1,z^q)$, and for $1\leq r<q$ the terms of
index $\nu\geq1$, $\nu=0$, and $\nu\leq-1$ of the bilateral series give
\[
f(z^r,z^{q-r})=S_r(z)+1+S_{q-r}(z).
\]
Summing over $r$ yields~\eqref{eq:theta}.
\end{proof}

For example,
\[
H_1(z)=z\psi(z),\qquad
H_2(z)=z\psi(z^2)+\frac z2\bigl(f(z,z)-1\bigr).
\]
Substituting~\eqref{eq:theta} into~\eqref{eq:aGF} gives the generating
function of every $a_q$. By the Jacobi triple
product~\cite[Chapter 16, Entry 19]{Berndt}, each theta function
in~\eqref{eq:theta} is an infinite product,
\[
f(z^r,z^{q-r})
=(-z^r;z^q)_\infty(-z^{q-r};z^q)_\infty(z^q;z^q)_\infty,
\]
where $(u;v)_\infty=\prod_{m\geq0}(1-uv^m)$.

\section{Nested recurrences}

\begin{theorem}\label{thm:nested}
For $q\geq1$, the function $C_q$ is the unique sequence with $C_q(n)=n$ for
$1\leq n\leq q$ and
\begin{equation}
 C_q(n)=q+C_q\bigl(n-C_q(n-1)\bigr)\qquad(n>q),
\end{equation}
where both arguments on the right lie in $\{1,\ldots,n-1\}$.
\end{theorem}

\begin{proof}
Exactly $r$ values of $a_q$ are at most $a_q(r)$, so
\begin{equation}\label{eq:invariant}
 C_q(a_q(r))=a_q(r)-r=q+C_q(r).
\end{equation}
For $n>q$, the number $r=n-C_q(n-1)$ is $1$ plus the number of
values of $a_q$ below $n$, so $a_q(r)$ is the first value that is at
least $n$. Since $a_q(1)=q+2\leq n+1$ and later steps add at most $2$,
either $a_q(r)=n$, or $a_q(r)=n+1$ and $n$ is missing. In both cases
$C_q(n)=C_q(a_q(r))$, and the recurrence follows from~\eqref{eq:invariant}.
Since $1\leq C_q(n-1)\leq n-1$, both arguments lie in $\{1,\ldots,n-1\}$,
and uniqueness follows by induction.
\end{proof}

The constant term $q$ in this recurrence is determined by the initial
value $x=q+2$. Equivalently, \eqref{eq:invariant} says that each
interval $(n,a_q(n)]$ contains exactly $q$ missing integers.

\subsection{General increments}

The same argument applies to miss increments greater than $2$.

\begin{proposition}\label{prop:generalnested}
Let $x\geq2$ and $k\geq1$ be integers. Let $g$ be the
$(0,x,1,k+1)$-hiccup sequence, and let
$U(n)$ count the positive integers at most $n$ outside $\im(g)$, with
$U(0)=0$. Then
\begin{equation}\label{eq:generaltransform}
 g(n)=n+x-k-1+kU(n).
\end{equation}
For $n\geq\max\{2,k,x-1\}$,
\begin{equation}\label{eq:generalnested}
 U(n)=x-k-1+
 \sum_{i=1}^{k}U\bigl(n-i+1-U(n-i)\bigr),
\end{equation}
and every argument on the right is a positive integer smaller than $n$.
\end{proposition}

\begin{proof}
Since $g(n)>n$, step $n\geq2$ is a hit if and only if $n\in\im(g)$. As
$1\notin\im(g)$, steps $2$ to $n$ contain $U(n)-1$ misses, and summing
the increments gives~\eqref{eq:generaltransform}.

For $t\geq1$ put $F(t)=U(t-U(t-1))$. As in the proof of
Theorem~\ref{thm:nested}, $t-U(t-1)$ is the index of the first value
of $g$ that is at least $t$. Hence $F(1)=1$, and $F$ increases by $1$
exactly at the integers
\[
 \delta_j=g(j-1)+1\qquad(j\geq2,\ j\notin\im(g)).
\]
The miss at step $j$ skips the $k$ integers
$\delta_j,\delta_j+1,\ldots,\delta_j+k-1$, and these pairwise disjoint
blocks, together with $1,\ldots,x-1$, make up the complement. For
$n\geq\max\{k,x-1\}$ we therefore have
\begin{align*}
 U(n)&=x-1+\sum_j\min\{k,\max\{0,n-\delta_j+1\}\},\\
 \sum_{i=0}^{k-1}F(n-i)
     &=k+\sum_j\min\{k,\max\{0,n-\delta_j+1\}\},
\end{align*}
where only finitely many summands are nonzero. Subtracting gives
\eqref{eq:generalnested}. If $n-i\geq1$, then
$1\leq U(n-i)\leq n-i$ and the argument $n-i+1-U(n-i)$ lies in
$[1,n-i]$. Otherwise $n=i=k$, and the argument is $1-U(0)=1<n$ because
$n\geq2$.
\end{proof}

For $x=k+1$ the constant in~\eqref{eq:generalnested} vanishes and
$U(n)=n$ for $1\leq n\leq k$. When $k\geq2$, these are the recurrence
and the initial values of the sequence $a_{0,k}$ of Ruskey and
Deugau~\cite[Equations (1) and (3)]{RD}, which counts the leaves at the
largest level of their labeled $k$-ary trees. Thus $U=a_{0,k}$, and
\eqref{eq:generaltransform} becomes their Theorem 31,
$g(n)=n+ka_{0,k}(n)$, where $g$ is denoted by $b_k$ in their Section 7.
For $k=2$, this is \seqnum{A045412}. We observed the relation with a
meta-Fibonacci recurrence jointly with Sloane, and Ruskey and Deugau
proved it. Proposition~\ref{prop:generalnested} proves the identity by
counting the skipped integers and applies to every initial value
$x\geq2$. For $a_q$ we have $k=1$, so the recurrence has a single term
and its constant is $q$.

\subsection{Trees}

The morphism of Fokkink and Joshi also leads to a tree model for all
the parameters of Proposition~\ref{prop:generalnested}. Recall that a
\emph{postorder} labeling of an ordered tree visits the subtrees of
each vertex from left to right and then the vertex itself.

\begin{proposition}\label{prop:tree}
Let $x\geq2$ and $k\geq1$ be integers, and let $K_h$ be the complete ordered
$k$-ary tree of height $h$, with $K_0$ a single leaf. Let $T_{-1}$ be a
single leaf and, for $h\geq0$, form $T_h$ by attaching to a new root,
in order, the tree $T_{h-1}$ followed by $x-2$ copies of $K_h$. Label
each $T_h$ in postorder, starting at $1$.

For every integer $n\geq1$, choose $h$ with $|T_h|\geq n$. The leaf
labels at most $n$ are exactly the positive integers at most $n$
missing from the $(0,x,1,k+1)$-hiccup sequence. Their number
is therefore $U(n)$, independently of the choice of $h$.
\end{proposition}

\begin{proof}
Encode leaves by $0$ and internal vertices by $1$. The postorder word
$W_h$ of $K_h$ satisfies $W_0=0$ and $W_{h+1}=W_h^k1$, so
$W_h=\phi^h(0)$ for the morphism $\phi(0)=0^k1$, $\phi(1)=1$. Since
$T_{h-1}$ is the first subtree of $T_h$, its vertices keep their labels
and their leaf status in $T_h$, and the postorder words of
$T_0,T_1,\ldots$ are successive prefixes of
\begin{equation}\label{eq:treeword}
 0\prod_{h\geq0}\bigl(W_h^{\,x-2}1\bigr).
\end{equation}
By the result of Fokkink and Joshi~\cite[Lemma 5, Equation (6)]{FJ},
the characteristic word of the $(0,x,1,k+1)$-hiccup sequence is
obtained from the fixed point of
\[
 \beta\longmapsto\beta0^{x-2}1,\qquad
 0\longmapsto0^k1,\qquad1\longmapsto1
\]
starting with $\beta$ by replacing $\beta$ with $0$. The
blocks appended at successive iterations are
$\phi^h(0^{x-2}1)=W_h^{\,x-2}1$, so this word is~\eqref{eq:treeword},
and its zeros are the missing integers.
\end{proof}

Each new root has $x-1$ children and every internal vertex of an
attached $K_h$ has $k$ children, so for $x=k+1$ all the trees are
$k$-ary. For $k=1$ and $x=q+2$ the attached trees are chains, and each
new root receives the previous tree followed by $q$ chains of
height $h$. For $x=2$ the trees $T_h$ are themselves chains, in
agreement with $U(n)=1$ for all $n\geq1$.

Ruskey and Deugau~\cite[Section 7, Lemma 30]{RD} describe the case
$x=k+1$ through the leftmost children of a tree read in preorder. Marking leaves
in postorder describes the same complement and extends to every initial value.

\section{Polygonal numbers}

\subsection{Partial complements}

For $q\geq1$, the $(q+2)$-gonal numbers, indexed from $0$, are
\begin{equation}
 P_q(m)=q\binom{m+1}{2}+m+1\qquad(m\geq0),
\end{equation}
so that $P_1,P_2,P_3,P_4$ are the triangular, square, pentagonal, and
hexagonal numbers, all starting at $1$. Kimberling and
Brown~\cite[Equation (7)]{KB} define $\widehat P_q$ by
$\widehat P_q(0)=1$ and
\[
 \widehat P_q(m)-\widehat P_q(m-1)
 =\floor{\frac{m+q-1}{q}}=\ceil{m/q}\qquad(m\geq1).
\]
For an increasing sequence $r$, write $r^*(1)<r^*(2)<\cdots$ for the
increasing enumeration of $\N\setminus\{r(m):m\geq0\}$. Theorem~\ref{thm:complement} then
gives
\begin{equation}\label{eq:KBidentification}
 c_q(j)=\widehat P_q(j-1),\qquad
 a_q=\widehat P_q^{\,*}.
\end{equation}

If moreover $r=(r(m))_{m\geq0}$ satisfies $r(0)=1$ and this complement
is infinite, the \emph{first partial complement} of $r$ is the sequence
$\mathcal T(r)=(u_m)_{m\geq0}$ defined by
\[
 u_0=1,\qquad u_{m+1}=r^*(u_m),
\]
where we keep the initial term $1$, as in the transformation $t$ of
Kimberling and Brown. They prove that
$\mathcal T(P_q)=\widehat P_q$~\cite[Theorem 2.3]{KB}, and the reverse
relation $\mathcal T(\widehat P_q)=P_q$ follows from their Theorems 2.1
and 2.3 and Corollary 4.2. We derive this relation directly from the
hiccup rule and give formulas for all iterates of $a_q$.

\begin{theorem}\label{thm:iterates}
For integers $q,n\geq1$ and $t\geq0$,
\begin{align}
 C_q\bigl(a_q^{\circ t}(n)\bigr)&=C_q(n)+qt,\label{eq:itercount}\\
 a_q^{\circ t}(n)&=n+tC_q(n)+q\binom{t+1}{2}.\label{eq:iterate}
\end{align}
In particular,
\begin{equation}
 a_q^{\circ t}(1)=P_q(t),\qquad
 \mathcal T(\widehat P_q)=P_q.
\end{equation}
\end{theorem}

\begin{proof}
Iterating~\eqref{eq:invariant} yields~\eqref{eq:itercount}.
By~\eqref{eq:aC}, the step from $a_q^{\circ h}(n)$ to
$a_q^{\circ(h+1)}(n)$ is $q+C_q(n)+qh$, and summing over $0\leq h<t$
gives~\eqref{eq:iterate}. With $n=1$ and $C_q(1)=1$ this is $P_q(t)$,
and by~\eqref{eq:KBidentification} the orbit of $1$ under
$a_q=\widehat P_q^{\,*}$ is $\mathcal T(\widehat P_q)$ by definition.
\end{proof}

For $q\leq4$ the orbits of $1$ are the following figurate numbers.
\begin{center}
\begin{tabular}{cll}
\toprule
Hiccup parameters & Orbit of $1$ & Figurate numbers\\
\midrule
$(0,3,1,2)$ & $1,3,6,10,15,21,\ldots$ & triangular\\
$(0,4,1,2)$ & $1,4,9,16,25,36,\ldots$ & square\\
$(0,5,1,2)$ & $1,5,12,22,35,51,\ldots$ & pentagonal\\
$(0,6,1,2)$ & $1,6,15,28,45,66,\ldots$ & hexagonal\\
\bottomrule
\end{tabular}
\end{center}

\subsection{Iteration}\label{sec:cloitre-nested}

For $q\geq1$ put
\[
 \Pi_q(k)=q\binom{k}{2}+k\qquad(k\geq0),
\]
so that $\Pi_q(0)=0$ and $\Pi_q(k)=P_q(k-1)$ for $k\geq1$, and define
\begin{equation}
 h_q(n)=\#\{k\geq1:\Pi_q(k)<n\},\qquad f_q(n)=n-h_q(n).
\end{equation}
The function $h_q$ takes the value $k$ on the $qk+1$ integers of
$[L_k,L_{k+1})$, where $L_k=\Pi_q(k)+1$, and inverting the quadratic
$\Pi_q$ we find
\begin{equation}
 f_q(n)=n-\floor{\frac{q-2+\sqrt{(q-2)^2+8q(n-1)}}{2q}}.
\end{equation}
In an earlier paper~\cite[Theorem 3.1 and Equation (2)]{CloitreNested}, we proved that
this function, with parameter $m=q$, is the solution with $f_q(1)=1$ of
\begin{equation}\label{eq:cloitre-nested}
 f_q(n+1)=n-f_q^{\circ q}(n)+f_q^{\circ(q+1)}(n).
\end{equation}
Moreover,
\begin{equation}
 f_q(n+1)=f_q(n)
 \quad\Longleftrightarrow\quad
 n=\Pi_q(k)=a_q^{\circ(k-1)}(1)\quad\text{for some }k\geq1,
\end{equation}
so $f_q$ repeats a value exactly after the $(q+2)$-gonal numbers. Its
$q$th iterate is expressed through $C_q$ in the next theorem.

\begin{theorem}\label{thm:fC}
For every $q,n\geq1$,
\begin{align}
 f_q^{\circ q}(n)&=n+1-C_q(n),\label{eq:fC}\\
 a_q(n)&=2n+q+1-f_q^{\circ q}(n).\label{eq:hiccup-f}
\end{align}
\end{theorem}

\begin{proof}
The case $n=1$ is immediate. For $n\geq2$ write $n=L_k+r$ with
$k\geq1$ and $0\leq r\leq qk$, and put
$\ell=\min\{q,1+\floor{r/k}\}$. On $[L_k,L_{k+1}-1]$ the map $f_q$
subtracts $k$, and on $[L_{k-1},L_k-1]$ it subtracts $k-1$, so
\[
 f_q^{\circ t}(L_k)=L_k-k-(t-1)(k-1)
 \qquad(1\leq t\leq q).
\]
These values lie in $[L_{k-1},L_k-1]$ and the last one is $L_{k-1}$
(for $k=1$ this interval is $\{1\}$, which $f_q$ fixes). As $f_q$ is
nondecreasing, the first $q$ iterates of $n$ do not fall below
$L_{k-1}$. The iterates $n,n-k,n-2k,\ldots$ remain in
$[L_k,L_{k+1})$ for $1+\floor{r/k}$ steps, so exactly $\ell$ of the
first $q$ steps subtract $k$ and the others subtract $k-1$, which gives
\begin{equation}\label{eq:fblock}
 f_q^{\circ q}(n)=n-q(k-1)-\ell.
\end{equation}
By~\eqref{eq:cclosed},
\[
 c_q\bigl(q(k-1)+1+t\bigr)
 =1+q\binom{k}{2}+tk\qquad(0\leq t\leq q),
\]
and the next missing integer is $c_q(qk+2)=L_{k+1}$. The missing
integers up to $L_k+r$ are therefore those of index at most
$q(k-1)+1+\ell$, so $C_q(n)=q(k-1)+1+\ell$. This is~\eqref{eq:fC},
and~\eqref{eq:aC} gives~\eqref{eq:hiccup-f}.
\end{proof}

The same computation also proves~\eqref{eq:cloitre-nested}.
By~\eqref{eq:fblock}, $f_q^{\circ q}(n)=L_k$ when $n=\Pi_q(k+1)$
(see also~\cite[Lemma 4.2]{CloitreNested}), and
$L_{k-1}\leq f_q^{\circ q}(n)\leq L_k-1$ when $L_k\leq n<\Pi_q(k+1)$,
since $f_q^{\circ q}$ is nondecreasing and equals $L_k-1$ at
$n=\Pi_q(k+1)-1$. In both cases, and trivially for $n=1$,
\[
 h_q\bigl(f_q^{\circ q}(n)\bigr)=h_q(n+1)-1,
\]
and substituting $f_q(u)=u-h_q(u)$ gives the recurrence.

Theorem~\ref{thm:fC} also gives an identity for
$b_q=P_q^{\,*}$, the sequence of positive integers that are not
$(q+2)$-gonal.

\begin{corollary}
For $q,n\geq1$,
\begin{equation}
 b_q^{\circ q}(n)=a_q(n)-1.
\end{equation}
\end{corollary}

\begin{proof}
Since $f_q(t)=1+\#\{k:b_q(k)<t\}$ (see~\cite[Section 6.1]{CloitreNested}),
we have, for positive integers $t$ and $n$,
\[
 f_q(t)\leq n\quad\Longleftrightarrow\quad t\leq b_q(n).
\]
Applied $q$ times, the equivalence also holds for $f_q^{\circ q}$
and $b_q^{\circ q}$. By~\eqref{eq:fC},
$f_q^{\circ q}(t)=1+\#\{k:a_q(k)\leq t\}$, so $t\leq b_q^{\circ q}(n)$
if and only if $t<a_q(n)$.
\end{proof}

For example, $b_2$ begins $2,3,5,6,7,8,10,\ldots$, and
$b_2(b_2(3))=b_2(5)=7=a_2(3)-1$. For $q=3$, $b_3^{\circ3}=a_3-1$, where
$b_3$ enumerates the non-pentagonal numbers.

Since $\Pi_q(k)\geq k$, we also have
\begin{align}
 f_q(n)&=\#\{k\in\N:k\leq n\leq \Pi_q(k)\},\\
 \sum_{n=1}^{N}f_q(n)
 &=\#\{(u,v)\in\N^2:v\leq u\leq \Pi_q(v),\ u\leq N\}.
\end{align}
These identities extend the square case in our earlier
paper~\cite[Section 6.2]{CloitreNested}.
Let $K_{d,q}$ be the generalized Connell sequence with modulus
$d\geq2$ and row lengths increasing by $q$, denoted $C_{m,r}$ with
$m=d$ and $r=q$ by Iannucci and Mills-Taylor~\cite{IM}. In their
formula (6), the denominator $2k$ should be $2r$. With this correction,
the formula becomes
\[
 K_{d,q}(n)=n+(d-1)\bigl(f_q(n)-1\bigr),
\]
since both sides start at $1$ and increase by $1$ after a
$(q+2)$-gonal index and by $d$ otherwise.

\subsection{Dispersions}

\begin{proposition}
For $q\geq1$, $j\geq1$, and $t\geq0$, set
\begin{equation}
 \mathcal D_q(j,t)=c_q(j)+jt+q\binom{t+1}{2}.
\end{equation}
Every positive integer occurs exactly once in this array, and
\begin{equation}\label{eq:rowsuccessor}
 a_q(\mathcal D_q(j,t))=\mathcal D_q(j,t+1).
\end{equation}
Equivalently, order the pairs $(u,v)\in\Nzero^2$ first by $u+qv$
and then by $v$. The rank of $(u,v)$, counted
from $1$, is
\begin{equation}\label{eq:rank}
 R_q(u,v)=c_q(u+1)+(u+1)v+q\binom{v+1}{2},
\end{equation}
and $a_q$ maps the rank of $(u,v)$ to the rank of $(u,v+1)$.
\end{proposition}

\begin{proof}
Since $C_q(c_q(j))=j$, Theorem~\ref{thm:iterates} gives
$\mathcal D_q(j,t)=a_q^{\circ t}(c_q(j))$, which
proves~\eqref{eq:rowsuccessor}. Every value of $a_q$ has a unique
preimage, which is smaller. Starting from any positive integer and
repeatedly taking preimages therefore ends at a unique missing integer
$c_q(j)$ after a unique number $t$ of steps. Hence every positive
integer occurs exactly once in the array.

For the ranks, the number of pairs with $u+qv<h$ is
$\sum_{s=0}^{h-1}(\lfloor s/q\rfloor+1)=c_q(h+1)-1$, and among the
pairs of weight $u+qv$ exactly $v$ precede $(u,v)$. The rank of $(u,v)$
is therefore $c_q(u+qv+1)+v$, which becomes~\eqref{eq:rank} after $v$
applications of~\eqref{eq:comprec}.
\end{proof}

For $q=3$ the dispersion begins
\[
 \begin{array}{c|rrrrr}
 j\backslash t&0&1&2&3&4\\\hline
 1&1&5&12&22&35\\
 2&2&7&15&26&40\\
 3&3&9&18&30&45\\
 4&4&11&21&34&50\\
 5&6&14&25&39&56\\
 6&8&17&29&44&62
 \end{array}
\]
The rank array of Kimberling and Brown~\cite[Section 1]{KB}, indexed
from $0$, is $a_q^{\mathrm{KB}}(i,j)=R_q(j,i)$ in our notation, so
$\mathcal D_q(j,t)=a_q^{\mathrm{KB}}(t,j-1)$. In their notation,
$\mathcal D_q=A(\widehat P_q)$, with first column $\widehat P_q$ and
first row $P_q$. Its rows are the orbits of $a_q=\widehat P_q^{\,*}$,
so it is the dispersion of $a_q$ in the terminology of their Section 1.

\section{Combinatorial interpretations}\label{sec:counting}

\subsection{Rows}

Take $q$ rows of length $1$, then $q$ rows of length $2$, and so on, and
label all positions consecutively from $1$. The first label of the
$r$th row of length $m$, where $1\leq r\leq q$, is
\[
 1+q\binom m2+(r-1)m=c_q\bigl(q(m-1)+r\bigr),
\]
so deleting the first position of every row leaves exactly the values
of $a_q$, in increasing order. Figure~\ref{fig:rows} arranges these
rows in the plane for $q\leq4$.

\begin{itemize}
\item For $x=3$, the rows of lengths $1,2,3,\ldots$ form a triangle
whose left side carries the missing labels.
\item For $x=4$, the two rows of length $m$ are a row and a column of
$m$ points that extend an $(m-1)\times m$ rectangle to an
$m\times(m+1)$ rectangle, and the first point of each is missing.
\item For $x=5$, the three rows of length $m$ are three consecutive
sides of a hexagonal ring of the triangular lattice, each side taken
without its last point, and their first points are missing. This is
the hexagonal arrangement described in our earlier
paper~\cite[Theorem 5.3 and Figure 3]{Cloitre}.
\item For $x=6$, the four rows of length $m$ are the four half-open
sides of the square ring $|u|+|v|=m$, which has $4m$ points, and the
missing labels occur at its four vertices on the axes.
\end{itemize}

\newcommand{\rowpicture}[1]{%
\begin{tikzpicture}[x=.57cm,y=.57cm,
 kept/.style={circle,draw=black!30,fill=white,minimum size=4.7mm,
 inner sep=0pt,font=\scriptsize},
 missing/.style={rectangle,draw=black,fill=black!12,minimum size=4.7mm,
 inner sep=0pt,font=\scriptsize}]
\foreach \m in {1,...,4}{
 \pgfmathtruncatemacro{\lastside}{#1-1}
 \foreach \r in {0,...,\lastside}{
  \pgfmathtruncatemacro{\lastpoint}{\m-1}
  \foreach \t in {0,...,\lastpoint}{
   \pgfmathtruncatemacro{\lab}{#1*\m*(\m-1)/2+\r*\m+\t+1}
   \ifcase#1\relax
   \or
    \pgfmathsetmacro{\xx}{\t-(\m-1)/2}
    \pgfmathsetmacro{\yy}{-(\m-1)*sqrt(3)/2}
   \or
    \ifnum\r=0
     \pgfmathsetmacro{\xx}{\t}\pgfmathsetmacro{\yy}{\m}
    \else
     \pgfmathsetmacro{\xx}{\m}\pgfmathsetmacro{\yy}{\m-\t}
    \fi
   \or
    \ifcase\r
     \pgfmathsetmacro{\uu}{-\m}\pgfmathsetmacro{\vv}{\t}
    \or
     \pgfmathsetmacro{\uu}{-\m+\t}\pgfmathsetmacro{\vv}{\m}
    \or
     \pgfmathsetmacro{\uu}{\t}\pgfmathsetmacro{\vv}{\m-\t}
    \fi
    \pgfmathsetmacro{\xx}{\uu+\vv/2}
    \pgfmathsetmacro{\yy}{sqrt(3)*\vv/2}
   \or
    \ifcase\r
     \pgfmathsetmacro{\xx}{\m-\t}\pgfmathsetmacro{\yy}{\t}
    \or
     \pgfmathsetmacro{\xx}{-\t}\pgfmathsetmacro{\yy}{\m-\t}
    \or
     \pgfmathsetmacro{\xx}{-\m+\t}\pgfmathsetmacro{\yy}{-\t}
    \or
     \pgfmathsetmacro{\xx}{\t}\pgfmathsetmacro{\yy}{-\m+\t}
    \fi
   \fi
   \ifnum\t=0
    \node[missing] at (\xx,\yy) {\lab};
   \else
    \node[kept] at (\xx,\yy) {\lab};
   \fi
  }
 }
}
\ifnum#1>1 \node[font=\scriptsize] at (0,0) {$0$};\fi
\end{tikzpicture}}

\begin{figure}[tbp]
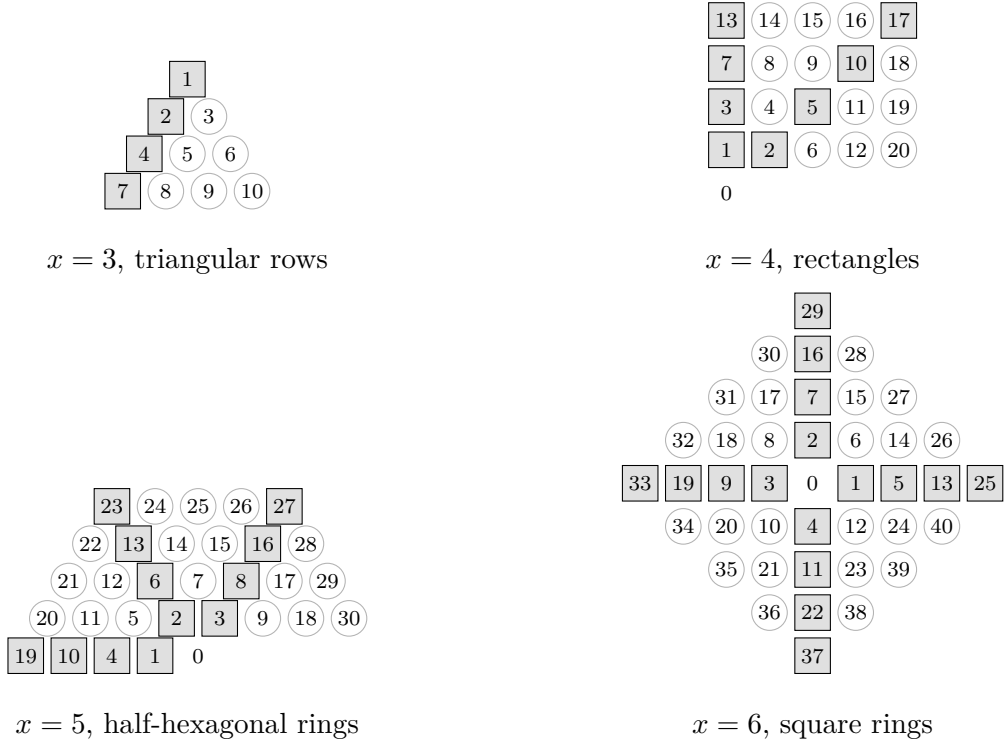

\centering
\begin{minipage}[b]{.47\linewidth}\centering
\rowpicture{1}\\[3mm] $x=3$, triangular rows
\end{minipage}\hfill
\begin{minipage}[b]{.47\linewidth}\centering
\rowpicture{2}\\[3mm] $x=4$, rectangles
\end{minipage}
\par\medskip
\begin{minipage}[b]{.47\linewidth}\centering
\rowpicture{3}\\[3mm] $x=5$, half-hexagonal rings
\end{minipage}\hfill
\begin{minipage}[b]{.47\linewidth}\centering
\rowpicture{4}\\[3mm] $x=6$, square rings
\end{minipage}
\caption{Rows of lengths $1$ to $4$ for the $(0,x,1,2)$-hiccup
sequences, $3\leq x\leq6$. Squares mark the first label of each row,
and the remaining labels are the values of the sequence. The point
labeled $0$ only marks the origin.}
\label{fig:rows}
\end{figure}

\subsection{Polyominoes}

The thumbtack interpretation of $a_2$ appears in the OEIS entry
\seqnum{A080037} and in our earlier paper~\cite[Theorem 5.2]{Cloitre}.
The same count is the least number of vertices of a polyomino with
$n$ cells. Harary and Harborth~\cite{HH} determined the extremal
numbers of vertices and edges of animals with a given number of cells.
See Malen, Rold\'an, and Toal\'a-Enr\'{\i}quez~\cite{MRT} for an account
and extensions. We recall the result and give a short proof. Here a
polyomino is a finite edge-connected union of unit squares of the square
lattice, possibly with holes.

\begin{proposition}
Among polyominoes with $n\geq1$ cells, the least number of vertices,
that is, of lattice points that are corners of some cell, is
\begin{equation}
n+1+\ceil{2\sqrt n}=a_2(n).
\end{equation}
\end{proposition}

\begin{proof}
Let $V$, $E$, $\mathcal P$, and $H$ be the numbers of vertices, unit
edges, boundary edges, and holes. Each cell has $4$ sides, and each
interior edge is shared by $2$ cells, so $E=2n+\mathcal P/2$.
Euler's formula $V-E+n=1-H$ gives
\[
V=n+1+\mathcal P/2-H.
\]
If the polyomino meets $h$ rows and $w$ columns, its outer boundary has
at least $2(h+w)$ edges and each hole boundary at least $4$, so
\[
V\geq n+1+h+w+H\geq n+1+\ceil{2\sqrt n}+H,
\]
where the last step uses $n\leq hw$ and the integrality of $h+w$.

Equality holds for the polyominoes obtained from an $m\times m$ square
by adding cells one at a time along a new row until an $m\times(m+1)$
rectangle is reached, and then along a new column until an
$(m+1)\times(m+1)$ square is reached. They have no holes, and their
half-perimeter is $2m$ at $n=m^2$, $2m+1$ for $m^2<n\leq m(m+1)$, and
$2m+2$ for $m(m+1)<n\leq(m+1)^2$, which is $\ceil{2\sqrt n}$ in each
case. Finally,
\[
n+1+\ceil{2\sqrt n}=\ceil{(1+\sqrt n)^2}=a_2(n)
\]
by Proposition~\ref{prop:radical}.
\end{proof}

\subsection{Monochromatic pairs}

For $q\geq1$ and $N\geq0$, let $E_q(N)$ be the least number of
monochromatic edges in a coloring of the vertices of $K_N$ with at most
$q$ colors. A coloring with classes of sizes $v_1,\ldots,v_q\geq0$ has
$\sum_i\binom{v_i}{2}$ monochromatic edges.

\begin{proposition}\label{prop:monochromatic}
For $q,j,n\geq1$,
\begin{align}
 E_q(j+q-1)&=c_q(j)-1,\\
 a_q(n)-n&=\min\{N\geq0:E_q(N)\geq n\}.
\end{align}
Thus $a_q(n)-n$ is the least number of vertices for which every
coloring with $q$ colors has at least $n$ monochromatic pairs.
\end{proposition}

\begin{proof}
If $v_i\geq v_{i'}+2$, moving a vertex from class $i$ to class $i'$
lowers the sum by $v_i-v_{i'}-1>0$, so the minimum is reached exactly
when all sizes differ by at most $1$. With $j-1=qm+r$ and $0\leq r<q$,
the balanced sizes for $N=j+q-1$ are $m+1$, taken $q-r$ times, and
$m+2$, taken $r$ times, hence
\[
 E_q(j+q-1)=(q-r)\binom{m+1}{2}+r\binom{m+2}{2}=c_q(j)-1.
\]
If $j=C_q(n)$, then $c_q(j)\leq n<c_q(j+1)$, that is,
$E_q(j+q-1)<n\leq E_q(j+q)$. Since $E_q$ is nondecreasing, the least
$N$ with $E_q(N)\geq n$ is $j+q$, which equals $a_q(n)-n$
by~\eqref{eq:aC}.
\end{proof}

The same argument is used in Tur\'an's theorem. The monochromatic
pairs are the edges missing from the complete $q$-partite graph.
For $q=3$ and $n=8$, we have $a_3(8)-8=9$. Three
classes of size $3$ contain $9$ monochromatic pairs, while on $8$
vertices the sizes $3,3,2$ give only $7$. Thus $9$ is the least number
of vertices that guarantees at least $8$ monochromatic pairs.

\section{Ramsey core numbers}\label{sec:ramsey}

A graph has a \emph{nonempty $t$-core} if it contains a nonempty
subgraph of minimum degree at least $t$. It has none if and only if its
vertices can be deleted one by one, each having degree at most $t-1$
when it is deleted, and such a graph is called
\emph{$(t-1)$-degenerate}.

For positive integers $t_1,\ldots,t_\ell$, the \emph{Ramsey core
number} $\rc(t_1,\ldots,t_\ell)$ is the least $N$ such that every
edge coloring of $K_N$ with $\ell$ colors has a nonempty $t_i$-core
in some color $i$. Bickle introduced and determined these
numbers~\cite{Bickle}. We express his formula through $C_1$ and $a_1$.

\subsection{Bickle's formula}

Put
\[
d_i=t_i-1,\qquad D=\sum_{i=1}^{\ell}d_i,\qquad
p=\sum_{1\leq i<j\leq\ell}d_id_j.
\]
A $d_i$-degenerate graph on $N\geq D+1$ vertices has at most
$d_iN-\binom{d_i+1}{2}$ edges, and for $N=D+r+1$ the number of edges
of $K_N$ exceeds the sum of these bounds by
\begin{equation}
\binom N2-\sum_{i=1}^{\ell}
\left(d_iN-\binom{d_i+1}{2}\right)=\binom{r+1}2-p.
\end{equation}
If $\binom r2\leq p<\binom{r+1}2$, this quantity is positive, so every
coloring of $K_{D+r+1}$ has a nonempty $t_i$-core in some color $i$.
Bickle~\cite[Theorems 34--36]{Bickle} shows, using the decomposition
theorem of Klein and Sch\"onheim~\cite{KS} (his Theorem 35), that
$K_{D+r}$ has a coloring with no such core. His formula, written in
terms of $\sum_i t_i$ and $\ell$, becomes
\begin{equation}\label{eq:bickle}
\rc(t_1,\ldots,t_\ell)
=D+\ceil{\frac12+\sqrt{2p+\frac94}}.
\end{equation}
The integer $r$ equals $C_1(p+1)$.

\begin{corollary}
With the notation above,
\begin{align}
\rc(t_1,\ldots,t_\ell)
 &=D+1+C_1(p+1)\label{eq:multiramseyC}\\
 &=a_1(p+1)+D-p-1.
\end{align}
For two colors, this becomes
\begin{equation}
\rc(s,t)=a_1\bigl((s-1)(t-1)+1\bigr)-(s-2)(t-2).
\end{equation}
\end{corollary}

\begin{proof}
Let $r=C_1(p+1)$. Since $c_1(j)=1+\binom j2$, we have
$\binom r2\leq p<\binom{r+1}2$, and as $p$ is an integer this gives
\[
r<\frac12+\sqrt{2p+\frac94}\leq r+1.
\]
Substituting in~\eqref{eq:bickle} yields the first identity,
and~\eqref{eq:aC} the second. For two colors, $D=s+t-2$ and
$p=(s-1)(t-1)$.
\end{proof}

For example, $\rc(3,t)=a_1(2t-1)-t+2$, which begins
$4,6,8,10,11,13,14,15,\ldots$ for $t\geq1$.

\subsection{Forests}

The identity $\rc(2,t)=a_1(t)$ is Bickle's
Theorem 38~\cite{Bickle}. It is also recorded in \seqnum{A080036} and
in our earlier paper~\cite[Theorem 5.1]{Cloitre}. We give a direct
proof.

\begin{theorem}
For $t\geq1$, let $r$ be the positive integer such that
$\binom r2<t\leq\binom{r+1}2$. Then
\begin{equation}
\rc(2,t)=t+r+1=a_1(t).
\end{equation}
\end{theorem}

\begin{proof}
Write $t=\binom r2+s$ with $1\leq s\leq r$. A forest on $N$ vertices
has at most $N-1$ edges and a $(t-1)$-degenerate graph at most
$(t-1)N-\binom t2$, and for $N=t+r+1$
\[
\binom N2-(N-1)-\left((t-1)N-\binom t2\right)=r-s+1>0,
\]
so $\rc(2,t)\leq t+r+1$.

For the lower bound, take a rooted tree whose internal vertices form a
path $v_1,\ldots,v_r$ from the root, where $v_i$ has $r-i+1$ children,
including $v_{i+1}$ when $i<r$. This tree has $1+\binom{r+1}2$
vertices, and adding $s-1$ isolated vertices brings the total to
$t+r$. Color the tree edges red and all other edges of $K_{t+r}$ blue.
The red graph is a forest, and when $v_1,\ldots,v_r$ are deleted in
this order each has blue degree $t-1$, after which only $t$ vertices
remain, so the blue graph is $(t-1)$-degenerate. Since $r=C_1(t)$, we
get $t+r+1=a_1(t)$.
\end{proof}

\subsection{A fixed sum}

The other complements $c_q$ appear when the sum $D$ is fixed and the
parameters $t_i$ vary. For $q\geq1$ and $D\geq q$, set
\[
\mathcal R_q(D)=
\max_{\substack{t_1,\ldots,t_q\geq2\\
                 \sum_{i=1}^q(t_i-1)=D}}
\rc(t_1,\ldots,t_q).
\]

\begin{corollary}
For integers $q\geq1$ and $D\geq q$, the maximum $\mathcal R_q(D)$
is attained when the $t_i$ differ by at most $1$, and
\begin{equation}
\mathcal R_q(D)
=D+1+C_1\!\left(\binom D2-c_q(D-q+1)+2\right).
\end{equation}
\end{corollary}

\begin{proof}
With $d_i=t_i-1$ and $D=\sum_id_i$ fixed,
\[
p=\sum_{i<j}d_id_j=\binom D2-\sum_i\binom{d_i}{2}
\]
is largest when the $d_i$ differ by at most $1$, and these balanced
values are positive because $D\geq q$. By
Proposition~\ref{prop:monochromatic},
\[
\max p=\binom D2-E_q(D)=\binom D2-c_q(D-q+1)+1,
\]
and the result follows from~\eqref{eq:multiramseyC}, since $C_1$ is
nondecreasing.
\end{proof}

For $q=3$ and $D=8$, the balanced parameters are $(4,4,3)$. Since
$c_3(6)=8$ and $C_1(22)=7$,
\[
\mathcal R_3(8)=9+C_1(28-8+2)=16=\rc(4,4,3).
\]

\end{document}